%% file: free5.tex
\documentclass[a4paper]{article}

\usepackage{graphicx}
\usepackage{color}
\usepackage{stmaryrd}

\usepackage{tikz}
\usepackage{tikz-cd}

\usepackage[numbers]{natbib}
\definecolor{mylinkcolor}{rgb}{0.05,0.05,0.4}
\usepackage[linktocpage]{hyperref}
\hypersetup{colorlinks,
linkcolor=mylinkcolor,
citecolor=mylinkcolor,
urlcolor=mylinkcolor}

\input{mac}

\input{thm}

\allowdisplaybreaks

\title{Properties of theories whose algebras are free}
\author{Tom Leinster%
\thanks{Maxwell Institute and School of Mathematics, University of
Edinburgh, Scotland; Tom.Leinster@ed.ac.uk}
\qquad
Maia Woolf\,%
\thanks{Maxwell Institute and School of Mathematics, University of
Edinburgh, Scotland; Maia.Woolf@ed.ac.uk}%
}
\date{}

\begin{document}

\sloppy
\maketitle

\begin{abstract}
A hard theorem of Givant classifies the finitary algebraic theories for
which every algebra is free. From it, one easily deduces that all such
theories are minimal (adding any further equations makes them trivial) and
have a well-defined notion of dimension (free algebras on sets of different
cardinality are never isomorphic). We prove both these properties directly,
without the aid of Givant's theorem.
\end{abstract}

\section{Introduction}
\label{sec:intro}

Which algebraic theories have the property that all their algebras are
free? In the 1970s, Steven Givant classified them all~(\cite{Giva},
p.~2). Roughly speaking, the only really nontrivial such theories are the
theory of vector spaces and some variants of it. (A precise statement is
given below.) And in the 2010s, Keith Kearnes, Emil Kiss and \'Agnes
Szendrei strengthened Givant's theorem, arriving at the same conclusion
with the weaker hypothesis that every \emph{finitely generated} algebra is
free~\cite{KKS}. Both proofs are substantial, in the latter case depending
on a long chain of previous theorems in universal algebra.

Once the classification theorem is proved (in either of its variants), one
can immediately read off certain further properties of theories whose
algebras are all free. But given the difficulty of the classification
theorem, it is desirable to find direct proofs of these properties too. 

Here we give direct proofs of two particular properties. Without appealing
to the classification theorem, we show that theories whose algebras are all
free are `minimal' and have a well-defined notion of dimension, in a sense
to be explained.

We now describe the results in more detail. By a \demph{theory} we mean a
finitary algebraic theory, and by an \demph{algebra} for a theory we mean
an algebra or model in the category of sets. We call a theory
\demph{totally free} if all its algebras are free.

The theory of vector spaces over a field $k$ is totally free, and the
same is true if $k$ is merely a division ring (skew field). The theory of
affine spaces (`vector spaces without origin') over a division ring is also
totally free. There are also four rather trivial totally free
theories, shown in Table~\ref{table:six}. Givant's theorem is that these
are the only totally free theories. (The top row of
Table~\ref{table:six} is invisible in both~\cite{Giva} and~\cite{KKS}, as
they follow the logicians' tradition of excluding trivial cases.)

\begin{table}
\centering
\begin{tabular}{l@{\hspace*{2em}}l}
Sets with at most one element   &
Sets with exactly one element   \\[1ex]
Sets                            &
Pointed sets                    \\[1ex]
Affine spaces over a division ring             &
Vector spaces over a division ring             
\end{tabular}
\label{table:six}
\caption{The totally free theories}
\end{table}

Some structure is apparent in Table~\ref{table:six}. In each row, the
right-hand column is the pointed version of the left-hand column: a set
with exactly one element is a pointed set with at most one element, and a
vector space is a pointed affine space (an affine space with a
distinguished element). This pointed/unpointed pattern is investigated in
Woolf's master's thesis~\cite{WoolATA}.

Here we prove, without the aid of the classification theorem, that
totally free theories are minimal (Section~\ref{sec:min}) and have a
well-defined notion of dimension (Section~\ref{sec:dim}). We now explain
briefly what these two results mean.

Informally, a nontrivial theory is \emph{minimal} if adjoining any new
equational law to it makes it trivial. (The formal definition is given in
Section~\ref{sec:min}.) For example, the theory of groups is \emph{not} minimal,
because adjoining the law $xy = yx$ produces the theory of abelian groups,
which is different from the theory of groups but not trivial. 

On the other hand, the theory of vector spaces is minimal, because
adjoining a law such as $x - y = 3x + 8y$ results in the trivial theory: a
vector space $V$ in which this equation holds for all $x, y \in V$ must be
trivial. This property of vector spaces is closely related to the fact that
the canonical map $V \to V^{**}$ to the double dual is injective
(Example~\ref{eg:min-vect}).

Our first main theorem is that totally free monads are minimal. In fact,
our result holds even for non-finitary monads, so it cannot quite be
deduced from the classification theorem.

Our second result on totally free theories is that if the free
algebras $F(X)$ and $F(Y)$ on sets $X$ and $Y$ are isomorphic, then $X$ and
$Y$ themselves are isomorphic. Thus, the dimension of an algebra $A$ can be
defined as the cardinality of the set $X$ such that $A \iso F(X)$. The
isomorphism classes of algebras therefore correspond one-to-one with the
isomorphism classes of sets, just as for vector spaces.

In general, theories do not have well-defined dimension in this sense. For
example, there are rings over which the free modules on $1, 2, \ldots$
generators are all isomorphic (Theorem~2D of~\cite{Ever}).

Categorically, a theory is totally free just when the free functor
$F$ from sets to algebras is essentially surjective on objects. For
dimension to be well-defined means that $F$ is essentially injective on
objects. Thus, essential surjectivity implies essential injectivity.

Although most theories are not minimal and do not have a notion of
dimension, these two properties are not strong enough (even together) to imply
that a theory is totally free. For instance, the theory of
semilattices is not totally free, but has both the minimality and
dimension properties (Examples~\ref{eg:min-semi} and~\ref{eg:dim-semi}).

\paragraph*{Conventions} We take a categorical approach, freely using the
correspondence between finitary algebraic theories and finitary monads on
$\Set$. Monads on $\Set$ (or their underlying endofunctors) are typically
denoted by $T$. 

The category of $T$-algebras is written as $\Set^T$, and the free functor
as $F \from \Set \to \Set^T$. The elements $\theta \in T(X)$ can be
understood as operations of arity $X$, for any set $X$. Correspondingly,
given a $T$-algebra $(A, T(A) \toby{\alpha} A)$, there is an induced map
\[
\theta_A \from A^X \to A.
\]
Indeed, for an $X$-indexed family $\family{a} = (a_x)_{x \in X} \in A^X$, we
can view $\family{a}$ as a function $X \to A$ and apply adjointness to
produce a homomorphism $\overline{\family{a}} \from F(X) \to (A, \alpha)$,
which is $\theta \mapsto \theta_A(\family{a})$. Assembling the functions
\[
\begin{array}{ccc}
T(X)    &\to            &A^{A^X}        \\
\theta  &\mapsto        &\theta_A
\end{array}
\]
($X \in \Set$), one shows that a $T$-algebra structure on a set $A$
amounts to a map of monads $T \to \End(A)$. Here $\End(A)$ is the
endomorphism monad on $A$, defined by $X \mapsto A^{A^X}$~\cite{KockDDM}. 

For a natural number $n$, we often use $n$ to also denote an $n$-element
set. For example, $F(n)$ means the free algebra on $n$ generators.

An algebra for a monad on $\Set$ is \demph{trivial} if it has at most one
element. A monad on $\Set$ is \demph{trivial} if all its algebras are
trivial. There are only two trivial monads. One is
\[
X \mapsto
\begin{cases}
\emptyset       &\text{if } X \text{ is empty},\\
1               &\text{otherwise}
\end{cases}
\]
($X \in \Set$). Its algebras are $\emptyset$ and the one-element set $1$,
and it corresponds to the theory presented by no operations and the single
equation $x = y$. The other is $X \mapsto 1$ for all sets $X$. Its only
algebra is $1$, and it corresponds to the theory presented by a single
constant and the equation $x = y$. These two trivial theories make up
the top row of Table~\ref{table:six}.

A monad $T$ on $\Set$ is nontrivial if and only if the unit map $X \to
T(X)$ is injective for every set $X$ (e.g.\ Exercise~2.3.11
of~\cite{BCT}). A further equivalent condition is that the free functor $F$
is faithful.

\paragraph*{Acknowledgements} Thanks to John Bourke, Steve Lack, Stefan
Milius and Lurdes Sousa for helpful comments. TL is grateful for the
generous donations by fellow members of the University and College Union
that paid his salary during a prolonged industrial dispute at the
University of Edinburgh, when this paper was written.

\section{Minimality}
\label{sec:min}

Here we define minimality, prove that every totally free theory is
minimal (unless trivial), then give further results and examples that
illustrate the significance of this theorem.

We will need a standard fact about monads on $\Set$ (Lemma~3.3.1
of~\cite{MaMu}). 

\begin{lemma}
\label{lemma:monics}
For every monad $T$ on $\Set$, both $T$ and the free functor $F$ preserve
monics. 
\end{lemma}

We assume throughout that $\Set$ satisfies the axiom of choice, otherwise
even the theory of vector spaces is not totally free. Thus, epics in
$\Set$ split, so $T$ and $F$ preserve epics too.

To define minimality, we first need to define quotients.  

Let $T$ be a monad on $\Set$. A \demph{quotient} of $T$ is another monad
$S$ on $\Set$ together with a map of monads $T \toby{\pi} S$ such that for
every set $X$, the function $\pi_X \from T(X) \to S(X)$ is
surjective. (Equivalently, $\pi$ is epic as a natural transformation.)
Syntactically, the theory $S$ is the theory $T$ together with some further
equational laws in the operations of $T$. Typical example: $T$ is the monad
for groups and $S$ is the monad for abelian groups.

We take the usual liberty of also referring to $S$ alone as a quotient of
$T$. We call $\pi$ a \demph{quotient map}, and denote quotient maps by
arrows $\epic$. 

\begin{example}
\label{eg:qt-End}
As noted in Section~\ref{sec:intro}, a $T$-algebra amounts to a set $A$
together with a map of monads $T \to \End(A)$. Any map of monads on $\Set$
factorizes uniquely as a quotient map followed by one that is componentwise
injective, in this case giving
\begin{equation}
\label{eq:qe-fact}
T \epic T_A \monic \End(A).
\end{equation}
(Showing that the standard factorization system on $[\Set, \Set]$ lifts to
the category of monads is straightforward, using the fact that monads on
$\Set$ preserve epics. Well known as this result seems to be, we have been
unable to find an explicit statement in the literature.) Here $T_A$ is the
\demph{theory of $A$}, which syntactically consists of the original theory
$T$ together with all equational laws satisfied in the algebra $A$:
operations $\theta, \theta' \in T(X)$ are identified in $T_A$ if and only
if $\theta_A = \theta'_A$.
\end{example}

Any map $T \toby{\pi} S$ of monads on $\Set$ induces a faithful functor
\[
\pi^*\from      \Set^S \to    \Set^T.
\]
If an $S$-algebra is regarded as a set $B$ together with a monad map $S
\to \End(B)$, then $\pi^*$ is simply composition with $\pi$. 

Assuming now that $\pi$ is a quotient map, $\pi^*$ is also full and
strictly injective on objects. This allows us to view $\Set^S$ as a full
subcategory of $\Set^T$, and to ask whether a $T$-algebra `is' an
$S$-algebra.

We can tell whether a quotient map $T \epic S$ is an isomorphism by
comparing the algebras for $T$ and $S$:

\begin{lemma}
\label{lemma:qt-iso}
Let $\pi \from T \epic S$ be a quotient map between monads on $\Set$. If
every $T$-algebra is an $S$-algebra then $\pi$ is an isomorphism.
\end{lemma}

\begin{proof}
Assuming that every $T$-algebra is an $S$-algebra, the inclusion $\pi^*
\from \Set^S \to \Set^T$ is strictly surjective on objects. It is in
any case full, faithful and strictly injective on objects, so it is an
isomorphism. But the functor
\[
\begin{array}{ccc}
(\text{monads on } \Set)^\op    &\to            &\CAT/\Set      \\
T                               &\mapsto        &(\Set^T \toby{U} \Set) \\
\rho                            &\mapsto        &\rho^*
\end{array}
\]
is full and faithful (by Theorem~6 of~\cite{StreFTM}) and therefore
reflects isomorphisms. Hence $\pi$ is an isomorphism.
\end{proof}

Birkhoff's HSP theorem characterizes, for a given $T$, the full
subcategories of $\Set^T$ that arise from quotients of $T$ (Theorem~3.3.4
of~\cite{ManeAT}, for instance). We will need the following fragment.

\begin{lemma}
\label{lemma:sp}
Let $S$ be a quotient of a monad $T$ on $\Set$. Then the full subcategory
$\Set^S$ of $\Set^T$ is closed under subalgebras and arbitrary products.
\end{lemma}

`Closed under subalgebras' means that any sub-$T$-algebra of an $S$-algebra
is an $S$-algebra.

\begin{proof}
This is clear syntactically, at least in the finitary case; for instance,
if a $T$-algebra satisfies the further equational laws required in order to
be an $S$-algebra, then so does every subalgebra. A proof for general
monads on $\Set$ can be found as part of Theorem~4.1 of~\cite{BarrHSP}. 
\end{proof}

A quotient map is \demph{proper} if it is not an isomorphism.  A monad on
$\Set$ is \demph{minimal} if it is nontrivial and has no proper
quotients. 

\begin{remark}
\label{rmk:qt-fin}
A quotient of a finitary monad is finitary. In terms of theories, this is
clear: when further equations are adjoined to a finitary algebraic theory,
it remains finitary. In terms of monads, it follows from the fact that an
endofunctor $T$ of $\Set$ is finitary if and only if, for every set $X$,
every element of $T(X)$ belongs to $T(Y)$ for some finite $Y \sub X$
(Proposition~C.5.2 of~\cite{AMM}). 
\end{remark}

Our proof that totally free monads are minimal rests on a simple
cardinality observation.

\begin{lemma}
\label{lemma:card}
Let $T$ be a monad on $\Set$ and let $A$ and $B$ be free $T$-algebras. If
$\mg{A} < \mg{B}$ then there is an injective homomorphism $A \to B$.
\end{lemma}

The hypothesis cannot be weakened to $\mg{A} \leq \mg{B}$, as can be seen
by considering finite-dimensional real vector spaces.

\begin{proof}
Write $A = F(X)$ and $B = F(Y)$. If $\mg{Y} \leq \mg{X}$ then there is an
injective homomorphism $F(Y) \to F(X)$, since $F$ preserves monics
(Lemma~\ref{lemma:monics}). But then $\mg{B} \leq \mg{A}$, a
contradiction. Hence $\mg{X} \leq \mg{Y}$. Again because $F$ preserves
monics, this implies that there is an injective homomorphism $A = F(X) \to
F(Y) = B$.
\end{proof}

\begin{thm}[Minimality]
\label{thm:min}
Every nontrivial totally free monad on $\Set$ is minimal.
\end{thm}

The theorem holds even for non-finitary monads, and so does not entirely
follow from Givant's theorem.

\begin{proof}
Let $T$ be a nontrivial totally free monad on $\Set$. Let $\pi
\from T \epic S$ be a quotient map, with $S$ nontrivial. We prove that $\pi$
is an isomorphism.

By Lemma~\ref{lemma:qt-iso}, it is enough to prove that every $T$-algebra
$A$ is an $S$-algebra. By Lemma~\ref{lemma:sp}, it is therefore enough to
prove that $A$ is a subalgebra of a product of $S$-algebras. 

Since $S$ is nontrivial, we can choose a nontrivial $S$-algebra $B$. And
since $B$ is nontrivial, there is a set $X$ such that $\mg{A} <
\mg{B}^{\mg{X}}$. By Lemma~\ref{lemma:card}, $A$ is a subalgebra of the
power $B^X$. Hence $A$ is a subalgebra of a product of $S$-algebras,
completing the proof.
\end{proof}

The rest of this section is not logically necessary, but sheds light on the
meaning of minimality and Theorem~\ref{thm:min}.

\begin{propn}
\label{propn:min-tfae}
For a nontrivial finitary monad $T$ on $\Set$, the following conditions are
equivalent. 
\begin{enumerate}
\item 
\label{part:mt-min}
$T$ is minimal;

\item
\label{part:mt-maps}
for all sets $X$, all distinct $\theta, \theta' \in F(X)$ and all
nontrivial $T$-algebras $A$, there exists a homomorphism $f \from F(X) \to
A$ such that $f(\theta) \neq f(\theta')$.

\item
\label{part:mt-ops}
for all nontrivial $T$-algebras $A$, all $n \in \N$, and all $\theta,
\theta' \in T(n)$, 
\[
\theta_A = \theta'_A 
\implies
\theta = \theta'.
\]
\end{enumerate}
\end{propn}

\begin{proof}
First we rephrase condition~\bref{part:mt-maps}. Homomorphisms $f \from
F(X) \to A$ correspond to maps of sets $X \to A$, which can be seen as
families $\family{a} = (a_x)_{x \in X} \in A^X$. Then $f(\theta) =
\theta_A(\family{a})$ for all $\theta \in F(X)$, as in
Section~\ref{sec:intro}. Condition~\bref{part:mt-maps} is therefore
equivalent to:
\begin{quote}
\textit{for all nontrivial $T$-algebras $A$, all sets $X$, and
all $\theta, \theta' \in T(X)$,}
\[
\theta_A = \theta'_A \implies \theta = \theta'.
\]
\end{quote}
We call this condition~(\ref{part:mt-maps}$'$).

\bref{part:mt-min}$\textimplies$(\ref{part:mt-maps}$'$):
assume~\bref{part:mt-min} and let $A$ be a nontrivial $T$-algebra. As in
Example~\ref{eg:qt-End}, there is a quotient map $\pi \from T \epic T_A$, and
the monad $T_A$ is nontrivial since it has $A$ as a nontrivial
algebra. Hence by~\bref{part:mt-min}, $\pi$ is an isomorphism. So by the
factorization~\eqref{eq:qe-fact}, the monad map 
\[
\begin{array}{ccc}
T       &\to            &\End(A)        \\
\theta  &\mapsto        &\theta_A
\end{array}
\]
is injective in each component, proving~(\ref{part:mt-maps}$'$).

(\ref{part:mt-maps}$'$)$\textimplies$\bref{part:mt-ops}: trivial. 

\bref{part:mt-ops}$\textimplies$\bref{part:mt-min}:
assume~\bref{part:mt-ops} and take a quotient $\pi \from T \epic S$. If
$\pi_n$ is injective for all finite sets $n$ then $\pi_X$ is injective for
all sets $X$, since $T$ is finitary. In that case, $\pi$ is an
isomorphism. 

Otherwise, $\pi_n(\theta) = \pi_n(\theta')$ for some $n \in \N$ and
distinct $\theta, \theta' \in T(n)$. Now on the one hand, $\theta_A =
\theta'_A$ for all $S$-algebras $A$, since $\pi_n(\theta) =
\pi_n(\theta')$. On the other, $\theta_A \neq \theta'_A$ for all nontrivial
$T$-algebras $A$, by~\bref{part:mt-ops}. Hence every $S$-algebra is
trivial; that is, $S$ is trivial.
\end{proof}

\begin{example}
\label{eg:min-gp}
The monad $T$ for groups is not minimal, since condition~\bref{part:mt-ops}
fails: the binary operations $\theta, \theta' \in T(2)$ representing $(x,
y) \mapsto xy$ and $(x, y) \mapsto yx$ induce the same function $A^2 \to A$
for any abelian group $A$, and of course nontrivial abelian groups exist.
\end{example}

\begin{example}
\label{eg:min-vect}
The monad $T$ for vector spaces over any division ring $k$ is minimal,
since condition~\bref{part:mt-maps} holds. Indeed, for any vector space
$V$, distinct $v, v' \in V$ and nontrivial vector space $W$, there exists a
linear map $f \from V \to W$ with $f(v) \neq f(v')$. Since every nontrivial
vector space contains a copy of the $1$-dimensional space $k$, the force of
this statement is that it holds when $W = k$. Thus, the minimality of the
theory of vector spaces essentially states that the canonical map $V \to
V^{**}$ into the double dual is always injective.
\end{example}

\begin{example}
\label{eg:min-semi}
The monad $T$ for join-semilattices is minimal. Again, we verify
condition~\bref{part:mt-maps}. Take a semilattice $B$ (which need not, in
fact, be free), distinct $\theta, \theta' \in B$, and a nontrivial
semilattice $A$. Without loss of generality, $\theta' \not\leq \theta$. We
can choose $a, a' \in A$ such that $a < a'$. There is a homomorphism $f
\from B \to A$ given by
\[
f(b) = 
\begin{cases}
a               &\text{if } b \leq \theta,      \\
a'              &\text{otherwise},
\end{cases}
\]
and then $f(\theta) \neq f(\theta')$, proving minimality.

However, the semilattice monad is not totally free. Indeed, $T(X)$ is
the set of finite subsets of $X$, for any set $X$, so on cardinality
grounds alone, the three-element totally ordered set is a semilattice that
is not free.
\end{example}

\section{Dimension}
\label{sec:dim}

Here we prove that every totally free nontrivial finitary monad on
$\Set$ has a well-defined notion of dimension, in the sense
that free algebras on sets of different cardinalities are never isomorphic.

When the cardinalities are infinite, this result is a generalisation of a standard result for vector spaces (Section~IX.2 of~\cite{Jaco}) and does not need total freeness. We begin by showing this.

\begin{lemma}
\label{lemma:union}
Let $T$ be a nontrivial finitary monad on $\Set$. Let $X$ and $Y$ be
sets. If there exists a surjective homomorphism $F(X) \epic F(Y)$ then $Y$
can be expressed as a union of $X$ finite subsets.
\end{lemma}

\begin{proof}
Take a surjective homomorphism $f \from F(X) \epic F(Y)$. For each $x \in
X$, the element $f(x)$ of $F(Y)$ belongs to $F(Y_x)$ for some finite subset
$Y_x$ of $Y$, since $T$ is finitary. Hence $\im f \sub F(\bigcup_{x \in X}
Y_x)$. But $f$ is surjective, so $F(\bigcup_{x \in X} Y_x) = F(Y)$. It
follows from the nontriviality of $T$ that $\bigcup_{x \in X} Y_x =
Y$. (This last step is clear syntactically. Alternatively, the image under
$F$ of $j \from \bigcup_{x \in X} Y_x \incl Y$ is an isomorphism; but $F$
is faithful since $T$ is nontrivial, so $F$ reflects epics and $j$
itself is epic.)
\end{proof}

\begin{propn}
\label{propn:inf-dim}
Let $T$ be a nontrivial finitary monad on $\Set$. Let $X$ and $Y$ be
infinite sets. If $F(X) \iso F(Y)$ as $T$-algebras then $X \iso Y$ as sets.
\end{propn}

\begin{proof}
This follows from Lemma~\ref{lemma:union} and elementary cardinal arithmetic.
\end{proof}

Proposition~\ref{propn:inf-dim} fails if we allow finite sets, as noted in
Section~\ref{sec:intro}. To extend it to arbitrary sets, we need to assume
that $T$ is totally free. But before making that assumption, we give two
further general results.

\begin{propn}
\label{propn:fg-free}
Let $T$ be a finitary monad on $\Set$. Then any finitely generated free
$T$-algebra is the free $T$-algebra on a finite set.
\end{propn}

\begin{proof}
If $T$ is trivial then so is the result; suppose not. If a free algebra
$F(X)$ is finitely generated then there is a surjective homomorphism $F(n)
\epic F(X)$ for some $n \in \N$, and Lemma~\ref{lemma:union} then implies
that $X$ is finite. 
\end{proof}

\begin{lemma}
\label{lemma:fg-colim}
Let $T$ be a finitary monad on $\Set$. Let $D\from \scat{I} \to \Set^T$ be
a filtered diagram such that the algebra $D(i)$ is finitely generated for
each object $i$ of $\scat{I}$ and the homomorphism $D(u) \from D(i) \to
D(j)$ is surjective for each map $u \from i \to j$ in $\scat{I}$. Let $n
\in \N$. If $\colim D$ can be generated by $n$ elements then so can $D(i)$
for some $i \in \scat{I}$.
\end{lemma}

The hypothesis that each algebra $D(i)$ is finitely generated cannot be
dropped, as demonstrated by the colimit
\[
\colim \bigl( \N \toby{f} \N \toby{f} \cdots \bigl) = 1
\]
in $\Set$, where $f(n) = \max(n - 1, 0)$. 

\begin{proof}
Write $\bigl(D(i) \toby{p_i} C = \colim D\bigr)_{i \in \scat{I}}$ for the
colimit cocone. By standard results on filtered colimits in categories of
algebras, the maps $p_i$ are jointly surjective, and two elements $a, b \in
D(i)$ satisfy $p_i(a) = p_i(b)$ if and only if $D(u)(a) = D(u)(b)$ for some
map $i \toby{u} j$ with domain $i$.

By hypothesis, $C$ is generated by some elements $c_1, \ldots,
c_n$. By filteredness, there exist $i \in \scat{I}$ and $c'_1, \ldots,
c'_n \in D(i)$ whose images under $p_i$ are $c_1, \ldots, c_n$
respectively. Since $D(i)$ is finitely generated, we can choose a finite
set $\{a_1, \ldots, a_m\}$ of generators. 

Let $r \in \{1, \ldots, m\}$. Since $p_i(a_r)$ belongs to $C$, which is
generated by $c_1, \ldots, c_n$, there is some $\theta^r \in T(n)$ such
that $p_i(a_r) = \theta^r_{C}(c_1, \ldots, c_n)$. Then
\[
p_i(a_r) 
= 
\theta^r_{C}\bigl(p_i(c'_1), \ldots, p_i(c'_n)\bigr) 
= 
p_i\bigl(\theta^r_{D(i)} (c'_1, \ldots, c'_n)\bigr),
\]
since $p_i$ is a homomorphism. Hence there exists a map $u_i \from i \to
j_i$ in $\scat{I}$ such that 
\[
D(u_i)(a_r) = D(u_i)\bigl(\theta^r_{D(i)} (c'_1, \ldots, c'_n)\bigr).
\]

This holds for all $r \in \{1, \ldots, m\}$, so by filteredness, there is a
\emph{single} map $u \from i \to j$ in $\scat{I}$ such that for all $r$,
\[
D(u)(a_r) = D(u)\bigl(\theta^r_{D(i)} (c'_1, \ldots, c'_n)\bigr).
\]
Thus, for all $r \in \{1, \ldots, m\}$, writing $d_s = D(u)(c'_s) \in D(j)$,
\[
D(u)(a_r) = \theta^r_{D(j)} (d_1, \ldots, d_n).
\]
In particular, $D(u)(a_1), \ldots, D(u)(a_m)$ all belong to the subalgebra
of $D(j)$ generated by $d_1, \ldots, d_n$. But $D(u)$ is surjective (by
hypothesis), and $a_1, \ldots, a_m$ generate $D(i)$, so $D(u)(a_1), \ldots,
D(u)(a_m)$ generate $D(j)$. Hence $d_1, \ldots, d_n$ also generate $D(j)$,
completing the proof.
\end{proof}

Our next result is the heart of the proof of the dimension theorem for
totally free theories. It rests on a striking application of Zorn's
lemma. We prove that a certain type of map is an isomorphism not by
constructing an inverse, but by merely showing that one exists.

\begin{propn}
\label{propn:self-surj}
Let $T$ be a totally free finitary monad on $\Set$, and let $n \in
\N$. Then every surjective homomorphism $F(n) \epic F(n)$ is an isomorphism.
\end{propn}

\begin{proof}
Let $P$ be the preordered set whose elements are surjective homomorphisms $F(n)
\epicby{s} F(n)$, with $s \leq s'$ if and only if there is a
homomorphism $t$ such that $ts = s'$:
\[
\begin{tikzcd}[column sep = tiny]
        &F(n) \ar[ld, "s"', two heads] \ar[rd, "s'", two heads]&       \\
F(n) \ar[rr, "t"', two heads]       &               &F(n).
\end{tikzcd}
\]
Then $t$ is both uniquely determined and surjective.

The elements of $P$ that are isomorphisms form a downwards closed
subset. Moreover, any maximal element of $P$ is an isomorphism:
indeed, every $s \in P$ satisfies $s \leq s^2$ by the commutative triangle
\[
\begin{tikzcd}[column sep = tiny]
        &F(n) \ar[ld, "s"', two heads] \ar[rd, "s^2", two heads]&       \\
F(n) \ar[rr, "s"', two heads]       &               &F(n),
\end{tikzcd}
\]
so if $s$ is maximal then $s$ is an isomorphism. Hence it suffices to prove
that every element of $P$ is less than or equal to some maximal
element. For this, we use Zorn's lemma.

The proof will be complete if we can show that every chain in $P$ has an
upper bound. The empty chain has $1_{F(n)}$ as an upper bound. Now take a
nonempty totally ordered set $I$ and an order-preserving map 
\[
\begin{array}{ccc}
I               &\to            &P      \\
i               &\mapsto        &s_i.
\end{array}
\]
We must prove there exists $s \in P$ such that $s_i \leq s$ for all $i \in
I$.

When $i \leq j$, write $t_{ij}$ for the unique homomorphism $F(n) \epic
F(n)$ such that
\[
\begin{tikzcd}[column sep = tiny]
        &F(n) \ar[ld, "s_i"', two heads] \ar[rd, "s_j", two heads]&       \\
F(n) \ar[rr, "t_{ij}"', two heads]       &               &F(n)
\end{tikzcd}
\]
commutes. Define a functor $D \from I \to \Set^T$ by $D(i) = F(n)$
and $D(i \leq j) = t_{ij}$. This is a filtered diagram, since $I$ is
nonempty and totally ordered. It has a colimit cocone $(D(i) \toby{p_i}
C)_{i \in I}$. Since $I$ is connected, there is a unique homomorphism $p
\from F(n) \to C$ such that
\[
\begin{tikzcd}[column sep = tiny]
        &[-1em] F(n) \ar[ld, "s_i"', two heads] \ar[rd, "p"]   &[+1em]       \\
D(i) = F(n) \ar[rr, "p_i"'{pos=.3}]       &                             &C
\end{tikzcd} 
\]
commutes for each $i \in I$.

Each homomorphism $D(i \leq j) = t_{ij}$ is surjective, and it follows from
the construction of filtered colimits that each of the coprojections $p_i$
is also surjective. Since $I$ is nonempty, $p$ is surjective and $C$ is
finitely generated. And $C$ is free too (since $T$ is totally free), so by
Proposition~\ref{propn:fg-free}, $C \iso F(m)$ for some $m \in
\N$. Lemma~\ref{lemma:fg-colim} then applies, giving a surjection $t \from
C \epic F(n)$.

Now consider the diagram
\[
\begin{tikzcd}
        &F(n) \ar[ld, "s_i"', two heads] 
\ar[d, "p", two heads] \ar[rd, "s = tp", two heads]    &        \\
F(n) \ar[r, "p_i"', two heads]  &
C \ar[r, "t"', two heads]&
F(n).
\end{tikzcd}
\]
Putting $s = tp \from F(n) \epic F(n)$ defines an element $s \in P$. For
each $i \in I$, the two inner triangles commute, so the outer triangle
commutes, giving $s_i \leq s$.
\end{proof}

\begin{propn}
\label{propn:free-surj}
Let $T$ be a totally free nontrivial finitary monad on $\Set$. Let
$X$ and $Y$ be sets. If there exists a surjective homomorphism $F(X) \epic
F(Y)$ then $\mg{X} \geq \mg{Y}$.
\end{propn}

\begin{proof}
When $X$ is infinite, this follows from Lemma~\ref{lemma:union}. Assume
that $X$ is finite, and suppose for a contradiction that $\mg{Y} >
\mg{X}$. Take a surjective homomorphism $f \from F(X) \epic F(Y)$. 

First note that $X$ is nonempty. For since $T$ is nontrivial, it has a
nontrivial algebra $A$, and since $Y$ is nonempty, we can find two distinct
homomorphisms $g, g'\from F(Y) \to A$. Since $f$ is epic, the homomorphisms
$gf, g'f \from F(X) \to A$ are also distinct. But $F(\emptyset)$ is the
initial algebra, so $X$ cannot be empty.

Since $\mg{Y} > \mg{X} > 0$, there is a noninjective surjection $p \from
Y \to X$. As $T$ is nontrivial, the homomorphism $F(p) \from F(Y) \to
F(X)$ is also noninjective, as well as surjective. Now the composite
\[
\begin{tikzcd}
F(X) \ar[r, "f", two heads] &F(Y) \ar[r, "F(p)", two heads] &F(X)
\end{tikzcd}
\]
is a surjective homomorphism $F(X) \epic F(X)$, with $X$ finite, and is
therefore an isomorphism by Proposition~\ref{propn:self-surj}. But this
implies that $f$ and $F(p)$ are isomorphisms, contradicting $F(p)$ being
noninjective.
\end{proof}

Proposition~\ref{propn:free-surj} immediately implies our second main
theorem: 

\begin{thm}[Dimension]
\label{thm:dim}
Let $T$ be a totally free nontrivial finitary monad on $\Set$. Let
$X$ and $Y$ be sets. If $F(X) \iso F(Y)$ as $T$-algebras then $X \iso Y$ as
sets. 
\end{thm}

When $T$ is totally free, Theorem~\ref{thm:dim} allows us to define
the \demph{dimension} of a $T$-algebra $A$ as the cardinality of the set
$X$, unique up to isomorphism, such that $A \iso F(X)$. Dimension is then a
complete invariant of $T$-algebras. 

\begin{example}
\label{eg:dim-semi}
Well-definedness of dimension is a weaker property than being
totally free. The theory of semilattices provides an example. Indeed,
we saw in Example~\ref{eg:min-semi} that it is not totally
free and that the corresponding monad is the finite powerset monad, from
which it follows by a cardinality argument that $F(X) \iso F(Y) \implies X
\iso Y$.
\end{example}

\bibliography{freerefs}

\end{document}

%% file: mac.tex
\usepackage{latexsym}
\usepackage{amsmath,amssymb}

\newcommand{\bref}[1]{(\ref{#1})}
\newcommand{\demphcolor}{\color[rgb]{0,0,0}}
\newcommand{\demph}[1]{\textbf{\textup{\demphcolor#1}}}

\newcommand{\fcat}[1]{\mathbf{#1}}
\newcommand{\scat}[1]{\mathbf{#1}}

\renewcommand{\epsilon}{\varepsilon}
\renewcommand{\phi}{\varphi}
\newcommand{\iso}{\cong}

\newcommand{\sub}{\subseteq}
\newcommand{\from}{\colon}

\newcommand{\textimplies}{\Rightarrow}

\newcommand{\N}{\mathbb{N}}

\renewcommand{\emptyset}{\varnothing}

\DeclareMathOperator{\colim}{colim} 
\DeclareMathOperator{\End}{End}
\DeclareMathOperator{\im}{im}
\newcommand{\op}{\mathrm{op}}

\newcommand{\CAT}{\fcat{CAT}}
\newcommand{\Set}{\fcat{Set}}

\newcommand{\mg}[1]{\lvert#1\rvert}

\newcommand{\toby}[1]{\xrightarrow{#1}}
\newcommand{\incl}{\hookrightarrow}
\newcommand{\monic}{\rightarrowtail}

\newcommand{\epic}{\twoheadrightarrow}
\newcommand{\epicby}[1]{\stackrel{#1}{\epic}}

\newcommand{\family}[1]{\mathbf{#1}}

%% file: thm.tex
\usepackage[amsmath,thmmarks]{ntheorem}

\newtheorem{thm}{Theorem}[section]
\newtheorem{propn}[thm]{Proposition}
\newtheorem{lemma}[thm]{Lemma}

\theorembodyfont{\normalfont}

\newtheorem{example}[thm]{Example}

\newtheorem{remark}[thm]{Remark}

\theoremstyle{nonumberplain}
\theoremsymbol{\ensuremath{\square}}
\qedsymbol{\ensuremath{\square}}

\newtheorem{proof}{Proof}

\newcommand{\theoremtobeproved}{}
\newtheorem{pfoftheorem}{Proof of \theoremtobeproved}